\documentclass[11pt,twoside]{article}

\usepackage{mathrsfs,amsfonts,amsmath,amssymb}
\usepackage{latexsym}
\usepackage{comment}

\newtheorem{satz}{Theorem}
\newtheorem{proposition}[satz]{Proposition}
\newtheorem{theorem}[satz]{Theorem}
\newtheorem{lemma}[satz]{Lemma}
\newtheorem{definition}[satz]{Definition}
\newtheorem{corollary}[satz]{Corollary}

\newtheorem{example}[satz]{Example}
\newtheorem{question}[satz]{Question}

\def\Z{\mathbb {Z}}
\def\R{\mathbb {R}}
\def\F{\mathbb {F}}
\def\E{\mathsf{E}}

\def\a{\alpha}

\def\C{\mathbb{C}}

\def\d{\delta}
\def\o{\omega}
\def\({\big (}
\def\){\big )}

\def\G{\Gamma}

\def\le{\leqslant}
\def\ge{\geqslant}
\def\_phi{\varphi}
\def\eps{\varepsilon}
\def\m{\times}
\def\Gr{{\mathbf G}}
\def\FF{\widehat}

\def\ov{\overline}
\def\Spec{{\rm Spec\,}}

\def\la{\lambda}
\def\D{\Delta}

\def\C{\mathbb{C}}

\def\cov{\mathsf{cov}}

\def\cov{{\rm cov}}
\def\un{{\rm un}}
\def\U{{\mathcal U}}

\newcommand{\bp}{\bigskip}

\author{I.D. Shkredov}
\title{
On universal sets and sumsets
}
\date{}
\begin{document}
	\maketitle


\begin{center}
	Annotation.
\end{center}

{\it \small
    We study the concept of universal sets from the additive--combinatorial point of view. 
    Among other results we obtain some applications of this type of uniformity to sets avoiding solutions to linear equations, and 
    get 
    an optimal upper bound for the covering number of general  sumsets.
}
\\

\section{Introduction}

Let $\Gr$ be an abelian group.
Given two sets $A,B\subseteq \Gr$, define  
the {\it sumset} 
of $A$ and $B$ as 
$$A+B:=\{a+b ~:~ a\in{A},\,b\in{B}\}\,.$$
In a similar way we define the {\it difference sets} and the {\it higher sumsets}, e.g., $2A-A$ is $A+A-A$.
The study of the structure of sumsets is a fundamental  problem in  classical additive combinatorics \cite{TV}, and this relatively new branch of mathematics provides  us with a number of excellent results on this object. 
For example, 
the Pl\"unnecke--Ruzsa inequality (see, e.g., \cite{Ruzsa_Plun} or \cite{TV}) says that for any positive integers $n$ and $m$ 
the following holds 
\begin{equation}\label{f:Pl-R} 
    |nA-mA| \le \left( \frac{|A+A|}{|A|} \right)^{n+m} \cdot |A| \,.
\end{equation} 
Thus, bound \eqref{f:Pl-R} connects the cardinality of the 
original 
set $A$ and its higher sumsets $nA-mA$. 
In this paper we do not consider cardinalities of sumsets, but rather deal with a new characteristic of a numerical set, namely universality and study how this characteristic changes from addition.
Namely, a set $A \subseteq \Gr$ is called the  {\it $k$--universal} set if  for any $x_1,\dots,x_k \in \Gr$ there exists $s\in \Gr$ such that $x_1+s,\dots,x_k+s \in A$.
The term ``universal set'' was coined in \cite{alon2009discrete}  and 
some applications to the Kakeya problem can be found in \cite{KLSS}.
If we denote the maximal $k$ with the above property  by $\un (A)$, then our first result says that 
\begin{equation}\label{f:un(A+B)_intr}
    \un (A+B) \ge \un (A) \un (B) \,.
\end{equation}
The estimate \eqref{f:un(A+B)_intr} is optimal and it shows, in particular, that the characteristic $\un (\cdot)$ grows  rapidly upon addition. 
A ``statistical'' or ``quantitative'' version of inequality \eqref{f:un(A+B)_intr} can be found in Section \ref{sec:universal} and can be considered as a generalization of the Ruzsa triangle inequality \cite{TV}.
The author hopes that such characteristics are interesting in its own right and can find application in further additive--combinatorial problems.
In the proofs we make extensive use of the concept of higher sumsets/higher energies, see \cite{SS_higher} and therefore 
the main idea 
of obtaining inequalities as \eqref{f:un(A+B)_intr} 
is to move into higher dimensions.

Any $k$--universal set has some random properties; universality  can be thought of as a combinatorial version of ordinary uniform sets (see, e.g., \cite{TV}).
For example, we show in Section \ref{sec:covering} that any $k$--universal subset $U$ of a finite group $\Gr$ intersects an arbitrary sumset $A+B$, where $|A| =\a |\Gr|$, $|B| = \beta |\Gr|$ and $\a \gg \frac{\log (1/\beta)}{k}$, say.
In particular, $U$ intersects any sufficiently long arithmetic progression or Bohr set.
Even stronger properties of universal sets can be found in Sections \ref{sec:Ek},  \ref{sec:covering} and \ref{sec:concluding}.
On the other hand, universality is much weaker property than classical uniformity (for example, {\it any} set $U\subseteq \Gr$ of size $|U| > (1-1/k) |\Gr|$ is $k$--universal) and  this gives the concept considerable flexibility. Using this flexibility, we obtain several applications.


In our first application we study sets $A\subseteq \Gr$, $|\Gr| <\infty$, $|A|=\d |\Gr|$ avoiding solutions to general linear equations 
\begin{equation}\label{f:linear_eq_intr}
    \a_1 x_1 + \dots + \a_n x_n = \beta \,,
\end{equation}
where  $n\ge 3$ and $\a_1,\dots,\a_n$ are coprime to $|\Gr|$. 
Denote by $f_A (x)$ the balanced function of the set $A$ (all required definitions can be found in Section  \ref{sec:def}). 
In the case of $n=3$, in the first part of the Kelly--Meka  proof \cite{kelley2023strong} it was shown that the so--called $\E_k$--norm $f_A$, $k=O(\log (2/\d))$ controls the number of solutions to the equation \eqref{f:linear_eq_intr}
in the sense that if $\| f_A\|_{\E_k}$ is small, then \eqref{f:linear_eq_intr} has $\Omega(\d^3 |\Gr|^2)$ solutions. For larger $n$ equation \eqref{f:linear_eq_intr} was studied in \cite{Gijswijt_systems}, \cite{Kosciuszko_eq}, \cite{Schoen_convex_eq} and in other papers.
Using the universality one can, in particular,  easily  show that for large $n$ one needs $\E_{k_*}$--norm of $f_A$ with much smaller $k_* = k_*(n)$ to control  the number of the solutions to the equation \eqref{f:linear_eq_intr}.

\begin{theorem}
    Let $\Gr$ be a finite abelian group, $n\ge 4$ be an integer, $A\subseteq \Gr$ be a set, $|A| = \d |\Gr|$ and $A$ has no solutions to the equation 
\begin{equation}\label{f:E_k_linear_new_cond_intr}
    \a_1 x_1 + \dots + \a_n x_n = \beta \,, \quad \quad x_j \in A,\, j\in [n] \,,
\end{equation}
    where $\a_1,\dots,\a_n$ are coprime to $|\Gr|$.
    Also, 
    let $\eps \in (0,1/8]$ be any number.
    Then there is $l \le (3\log (1/\d) )^{\frac{1}{[n/2]-1}}+1$ such that 
\begin{equation}\label{f:E_k_linear_new_intr}
    \|f_A \|^{2l}_{\E_l} \ge \eps^{2l} \d^{2l} |\Gr|^{l+1} \,.
\end{equation}
\label{t:E_k_linear_new_intr}
\end{theorem}

The second area of our applications concerns the so--called covering numbers.
Let $\Gr$ be a finite abelian group and $A\subseteq \Gr$ be  a set. 
The {\it covering number} of $A$ is defined as 
\begin{equation}\label{def:covering}
    \cov (A) = \min_{X\subseteq \Gr} \{ |X| ~:~ A+X = \Gr \} \,.
\end{equation}
In other words, the covering number of $A$ is the minimum  number of translates
$A$ required to cover the group. 
In such a generalized form this concept can be found in
\cite{bollobas2011covering} (see also previous papers \cite{newman1967complements}, \cite{schmidt2003complementary}), which contains many references to this issue. The case of an infinite group $\Gr$ (e.g. $\Z^d$ or $\R^d$) is different, and the interested reader may refer to, for example, the classic book \cite{rogers1964packing}.


In this generality, the covering number $\cov(A)$ can be considered as a characteristic of the randomness of $A$. 
Indeed,   for any set $A$ one has $\cov(A) \ge |\Gr|/|A|$ and the last bound is attained if, say, $\Gr = \Z/N\Z$, $|A|$ divides $N$ and $A$ is an arithmetic progression $\{1,\dots, |A|\}$.  
On the other hand, 
using the simplest greedy algorithm,  one can show (see, e.g., \cite[Corollary 3.2]{bollobas2011covering}) that 
\begin{equation}\label{f:basic_bounds}
    \frac{|\Gr|}{|A|} \le \cov (A) \le \frac{|\Gr|}{|A|} \left( \log |A| +1 \right) < \frac{|\Gr|}{|A|} \log |\Gr| 
\,.
\end{equation}
The upper bound in \eqref{f:basic_bounds} is tight and attained at a random set $A$. 
Thus one would expect that $\cov (A)$ should be close to the lower bound in \eqref{f:basic_bounds} for structured sets as arithmetic progressions and should be close to the upper bound if $A$ has some random properties.


A classic result of additive combinatorics, namely the Ruzsa covering lemma (see \cite{ruzsa1999analog} and \cite[Lemma 2.14]{TV}, confirms this heuristic and gives us a good upper bound for the covering number of difference sets considered in additive theory as a typical example of structural sets.

\begin{theorem}
    Let $\Gr$ be a finite abelian group, $A \subseteq \Gr$ be a set, $|A|=\a |\Gr|$.
    Then
\begin{equation}\label{f:Ruzsa_intr}
    \cov (A-A) \le \a^{-1} \,.
\end{equation}
\label{t:Ruzsa_intr}
\end{theorem}

We obtain a 
generalization of Theorem \ref{t:Ruzsa_intr}, namely, 
we show that a similar result takes place for an 
arbitrary sumset $A+B$. The method of the proof differs  from \cite{ruzsa1999analog}.

\begin{theorem}
    Let $\Gr$ be a finite abelian group, $A,B \subseteq \Gr$ be sets, $|A| = \a |\Gr|$, and $|B| = \beta |\Gr|$.
    Then 
\begin{equation}\label{f:cov_AB_intr}
    \cov (A+B) \le \frac{1}{\a} \log \frac{1}{\beta} + 1 \,.
\end{equation}
\label{t:cov_AB_intr}
\end{theorem}


Bound \eqref{f:cov_AB_intr} in Theorem \ref{t:cov_AB_intr} is absolutely tight, see Section \ref{sec:covering}.  It is also interesting to note that while an infinite analogue of Theorem \ref{t:Ruzsa_intr} takes place, see  \cite{furstenberg1977ergodic}, \cite{stewart1979infinite} but the correspondent infinite version of Theorem \ref{t:cov_AB_intr} for infinite sets fails (see, e.g., \cite{stewart1979infinite}).

\bp

Our main new observation is the following: the {\it complement} $A^c$ of any set $A$ with large $\cov(A)$ is a universal set, namely, we have the basic formula 
\begin{equation}\label{f:cov_un}
    \cov (A) = \un (A^c) + 1 \,,
\end{equation}
and therefore $A^c$ must have some random/repulsive properties. On the other hand, the complements to sumsets of dense sets  are known to be additively rich sets (see, e.g., \cite{croot2012some}, \cite{Semchankau_wrappers}, \cite{SSS}) and this contradiction  gives us that the required upper bound for the covering number of any sumset.

In our last result, the group $\Gr$ 
coincides with 
the prime field $\F_p$, and we study two covering numbers $\cov^{+}/\cov^\times$ with respect to addition/multiplication  in $\F_p$. 
As the reader can see we are able to calculate the covering numbers for the
sumsets/product sets of almost all possible  types in $\F_p$. 
Given sets $A,B \subseteq \Gr$  and $\eps \in (0,1)$ we write $A+_\eps B$ for the set of $x\in \Gr$, having at least $\eps |\Gr|$ representations as sum of $a+b$, $a\in A$, $b\in B$.

\begin{theorem}
\label{t:table_intr}
    Let $p$ be a prime number, and $A, B\subset \F_p$ be sets.
    Suppose that $A,B$ and all sets below have positive density.
    Write $L$ for $\log p$. Then
\[
\]    
\begin{tabular}{|c|c|c|c|c|c|c|c|c|} 
\hline  
\textbf{}  & $A-A$ & $(A-A)^c$ & $A+B$ & $(A+B)^c$ & $AB$ & $(AB)^c$ & $(A+B)^{-1}$ & $((A+B)^{-1})^c$\\ 
\hline  
$\cov^{+}$ & $O(1)$ & $\forall_{O(1)}$ & $O(1)$ & $\forall_{O(1)}$ & $?_{\Omega(L)}$ & $\Omega(L)$ & $?_{\Omega(L)}$ & $\Omega(L)$ \\  
\hline  
$\cov^{\times}$ & $O(1)$ & $\Omega(L)$ & $\forall$ & $\forall$ & $O(1)$ & $\forall_{O(1)}$ & $\forall$ & 
$\forall$ 
\\  
\hline  
\end{tabular}\\ 
    Here the symbol $\forall$ means that the correspondent covering number can be 
    $O(1)$ or $\Omega (L)$.\\ 
    The symbol $\forall_{O(1)}$ means that the covering number of  
    $(A+B)^c$, say, can be arbitrary but the covering number of $(A+_{\Omega(1)} B)^c$ is $O(1)$.\\  
    Similarly, the question mark $?_{\Omega(L)}$ means that the covering number of  
    $AB$, say, is unknown  but the covering number of $A \cdot_{\Omega(1)} B$ is $\Omega(L)$.  
\end{theorem}


It follows from Theorem \ref{t:table_intr} that there are sets $S$, $|S|=o(p)$ such that 
\[\max\{ \cov^{+} (S), \cov^{\times} (S) \} = O(1) \,.\] 
Indeed, take $S=A-A$, $|S|=o(p)$,  where $A\subseteq \F_p$ is a set of large density. Thus, an analogue of the sum--product phenomenon (see, e.g., \cite{TV}) {\it has no place} for the covering numbers.






\section{Definitions and notation}
\label{sec:def}

Let $\Gr$ be a finite abelian group and we denote the cardinality of $\Gr$ by $N$. 
Given two sets $A,B\subseteq \Gr$, define  
the {\it sumset} 
of $A$ and $B$ as 
$$A+B:=\{a+b ~:~ a\in{A},\,b\in{B}\}\,.$$
Similarly, if $A,B \subseteq \F_p$ we can define the {\it product set} and the {\it ratio set} 
$$AB:=\{ab ~:~ a\in{A},\,b\in{B}\}\,, 
    \quad \quad 
    A/B:=\{a/b ~:~ a\in{A},\,b\in{B}\setminus \{0\} \}\,, 
$$
Given sets $A,B \subseteq \Gr$  and $\eps \in (0,1)$ we write $A+_\eps B$ for the set of $x\in \Gr$, having at least $\eps |\Gr|$ representations as sum of $a+b$, $a\in A$, $b\in B$.
Sometimes such subsets of $A+B$ are called {\it popular} sumsets. 
Also, denote by $A^c$ the complement of $A\subseteq \Gr$, that is $A^c = \Gr \setminus A$.  
If $A\subseteq \F_p$ and $\lambda \in \F_p$, $\la \neq 0$, then we write $\lambda \cdot A := \{\lambda a ~:~ a\in A\}$.
Given a set $A\subseteq \Gr$ and a positive integer $k$, we put 
$$
    \Delta_k (A) := \{ (a,a, \dots, a) ~:~ a\in A \} \subseteq \Gr^k \,.
$$
Also, let 
$\Delta_k (x) := \Delta_k (\{ x \})$, $x\in \Gr$.
Now we have 
\begin{equation}\label{def:A-A_intersection}
    A-A := \{ a-b ~:~ a,b\in A\} = \{ s\in \Gr ~:~ A\cap (A+s) \neq \emptyset \} \,.
\end{equation}
A natural generalization of the last formula in \eqref{def:A-A_intersection} 
is the set 
\begin{equation}\label{def:A-A_m_intersection}
    \{ (x_1, \dots, x_k) \in \Gr^k ~:~ A\cap (A+x_1) \cap \dots \cap (A+x_k) \neq \emptyset \} 
    =
    A^k - \Delta_k (A) \,,
\end{equation}
which is called the {\it higher difference set} (see \cite{SS_higher}). 
A simple characterization of the higher difference sets is given in \cite[Proposition 10]{SS_dim}. 

\begin{lemma}
    Let $k \ge 2$, $m \in [k]$ be positive integers, and
    let $A_1,\dots,A_k,B \subseteq \Gr$ be finite subsets of an abelian group $\Gr$.
    Then
    \begin{equation}\label{f:characteristic2}
        A_1 \m \dots \m A_k - \Delta_k (B)
            =
    \end{equation}
    $$
                \bigsqcup_{(x_1,\dots,x_m) \in A_1 \m \dots \m A_m - \Delta_m (B)}
                    \{ (x_1,\dots,x_m) \} \m (A_{m+1} \m \dots \m A_k - \Delta_{k-m}(B \cap (A_1-x_1) \cap \dots \cap (A_m-x_m)) \,.
    $$
\label{l:characteristic}
\end{lemma}

In particular, having a set  $A\subseteq \Gr$ and another  set $X=\{x_1,\dots, x_k\}$ and  writing  
$$A_X = (A+x_1) \cap \dots \cap (A+x_k) \,,$$ 
we see that 
\[
    A^k - \Delta_k (A) = \{ X \subseteq \Gr^k ~:~ A\cap A_X \neq \emptyset \} 
    \,.
\]
Observe 
the basic 
inclusion  
\begin{equation}\label{f:KK_inclusion}
    A_X -X \subseteq A \,.
\end{equation}

We need 
the generalized triangle inequality \cite[Theorem 7]{SS_higher}.

\begin{lemma}
    Let $k_1,k_2$ be positive integers, $W\subseteq \Gr^{k_1}$, $Y\subseteq \Gr^{k_2}$ and $X,Z \subseteq \Gr$. 
    Then
\begin{equation}\label{f:gen_triangle_S}
    |W\times X| |Y-\Delta_{k_2} (Z)| \le |W\times Y \times Z - \Delta_{k_1+k_2+1} (X)| \,, 
\end{equation}
    and 
\begin{equation}\label{f:gen_triangle_S'}
    |W\times Z - \Delta_{k_1+1} (X)| = |W \times X - \Delta_{k_1+1} (Z) | \,. 
\end{equation}
    In particular, for any $k\ge 2$ and sets $A_1,\dots, A_k \subseteq \Gr$ one has 
\begin{equation}\label{f:gen_triangle_G}
    |A_1 \times \dots \times A_k - \D_k (\Gr)| = |\Gr| |A_1 \times \dots \times A_{k-1} - \D_{k-1} (A_k)| \,. 
\end{equation}
\label{l:gen_triangle_S}
\end{lemma}

Let $\FF{\Gr}$ be its dual group in representation theory.
For any function $f:\Gr \to \mathbb{C}$ and $\rho \in \FF{\Gr}$ define 
the Fourier transform of $f$ at $\rho$ by the formula 
\begin{equation}\label{f:Fourier_representations}
\FF{f} (\rho) = \sum_{g\in \F_p} f(g) \ov{\rho (g)} \,.
\end{equation}
The Parseval identity is 
\begin{equation}\label{F_Par}
    N\sum_{g\in \Gr} |f(g)|^2
        =
            \sum_{\rho \in \FF{\Gr}} \big|\widehat{f} (\rho)\big|^2 \,.
\end{equation}
The Wiener norm of a function $f:\Gr \to \C$ is 
\begin{equation}\label{def:Wiener}
    \| f\|_W := N^{-1} \sum_{\rho \in \FF{\Gr}} |\FF{f} (\rho)| \,.
\end{equation}
We use the same capital letter to denote a set $A\subseteq \Gr$ and   its characteristic function $A: \Gr \to \{0,1 \}$. 
Given a set $A\subseteq \Gr$ and $\eps \in (0,1]$ define the {\it spectrum} of $A$ as 
\begin{equation}\label{def:spectrum}
    \Spec_\eps (A) = \{ \chi \in \Gr ~:~ |\FF{A} (\chi)|\ge \eps |A| \} \,.
\end{equation}
We write $\Spec'_\eps (A)$ for $\Spec_\eps (A) \setminus \{ 1\}$.
If $f,g : \Gr \to \C$ are some functions, then 
$$
    (f*g) (x) := \sum_{y\in \Gr} f(y) g(x-y) \quad \mbox{ and } \quad (f\circ g) (x) := \sum_{y\in \Gr} f(y) g(y+x) \,.
$$
For a real function $f$ put 
\[
    \| f\|^{2k}_{\E_k} = \sum_{x} (f\circ f)^k (x) \ge 0 \,.
\]
In particular, we obtain the {\it higher energy} (see \cite{SS_higher})  
\[
      \E_{k}(A) =  \sum_{x} (A\circ A)^k (x) \,.
\]

Consider a refinement of definition \eqref{def:covering} which allows to say something about the set of shifts $X$. Namely, let $f:\Gr \to \C$ be a function and $\eps \in (0,1]$ be a parameter. Define 
\begin{equation}\label{def:covering_f_eps}
    \cov_{f,\eps} (A) = \min_{X\subseteq \Gr} \{ |X| ~:~ A+X = \Gr\,,~ |\FF{X*f} (\chi)| \le \eps |X| \| f\|_1, \forall \chi\neq 1 \} \,.
\end{equation}
Clearly, $\cov_{} (A) \le \cov_{f,\eps} (A)$.
Further, assuming that $\eps \ge N^{-1/3} \log N$ and taking a random set $X$, $|X| \gg \eps^{-2} \log^2 N$ one can show that 
\begin{equation}\label{f:covering_f_eps_trivial}
    \cov_{f,\eps} (A) \ll \max\left\{ \frac{N}{|A|} \cdot  \log N, \eps^{-2} \log^2 N \right\} \,,
\end{equation}
see details in the  proof of Theorem \ref{t:cov_ABC} below.

\bp

We need a special case of \cite[Proposition 15]{SSS} (the best results are contained in \cite{Semchankau_wrappers}, and the first  result of this type was obtained in \cite{croot2012some}), see also  \cite[Section 6]{SSS} for general $\Gr$. 
Recall that for a given subset $\G \subseteq \FF{\Gr}$ and a number  $\eps \in (0,1)$ the set $\mathcal{B}(\Gamma, \eps)$ is called 
{\it Bohr set} (see, e.g., 
\cite[Section 4.4]{TV}), if
\[
	\mathcal{B}(\Gamma, \eps) = \{ x\in \Gr  ~:~ | \chi(x) -1| \le \eps \,, \forall \chi \in \G \} \,.
\]
The size of $\G$ is called the {\it dimension} of $\mathcal{B}(\Gamma, \eps)$, $\eps$ is the {\it radius} of $\mathcal{B}(\Gamma, \eps)$. 
The connection between the size of the Bohr set, its dimension and radius is well--known, see, for example, \cite[Lemma 4.20]{TV}
\begin{equation}\label{f:Bohr_size}
	|\mathcal{B}(\Gamma, \eps)| \ge (\eps/2\pi)^{|\G|} |\Gr|  \,.
\end{equation}

\begin{theorem}
    Let $\Gr$ be an abelian group,
    $\d \in (0,1)$ be a parameter, $A,B\subseteq \Gr$ be sets, $|A|, |B| \gg |\Gr|$ and $|(A+B)^c| \gg |\Gr|$.
    Then the complement to $A+_\delta B$ contains a shift of a  Bohr set $\mathcal{B}$ of dimension $O_\d (1)$ and the radius $\Omega_\d (1)$. 
\label{t:croot2012some}
\end{theorem}


Finally, we need \cite[Theorem 5.4]{s_Fish} to fill in some 
gaps 
in the proof of Theorem \ref{t:table_intr}. 

\begin{lemma}
    Let $q$ be a positive integer, $A\subseteq \Z/q\Z$ be a set, $|A|=\a q$. Suppose that the least prime factor of $q$ greater than $2\a^{-1} +3$. Then 
    $\cov^\times (A-A) \le \a^{-1}+1$. 
\label{l:Fish_covering}
\end{lemma}


The signs $\ll$ and $\gg$ are the usual Vinogradov symbols. 
When the constants in the signs  depend on a parameter $M$, we write $\ll_M$ and $\gg_M$.
Let us denote by $[n]$ the set $\{1,2,\dots, n\}$.
All logarithms are to base $e$.
For a prime number $p$ we write $\F_p = \Z/p\Z$ and $\F^*_p = \F_p \setminus \{0\}$.

\section{On universal sets}
\label{sec:universal}

In this section we discuss 
properties of universal sets and obtain a generalization of Ruzsa's triangle inequality. 
Having a set $A\subseteq \Gr$ we put 
\begin{equation}\label{def:un}
    \un(A) = \un (A,\Gr) = \max\{ k \ge 1 ~:~ A^{k} - \D_{k} (\Gr) = \Gr^{k} \} \,.
\end{equation}
Alternatively, since $A-\D_1 (\Gr) = \Gr$, we have in view of formula \eqref{f:gen_triangle_G} of Lemma \ref{l:gen_triangle_S} that 
\begin{equation}\label{def:un'}
    \un(A) = \max \{1, \max\{ k \ge 2 ~:~ A^{k-1} - \D_{k-1} (A) = \Gr^{k-1} \} \} \,.
\end{equation}
For example, $\un(A)=2$ iff $A-A = \Gr$. 
If there is no such $k$, then we define $\un (A) = 1$ appealing to the first definition \eqref{def:un}. 
The statement $A^{k-1} - \D_{k-1} (A) = \Gr^{k-1}$  is equivalent to say that $A$ is a $k$--universal set, 
that is for any $z_1,\dots,z_k \in \Gr$ there is $w\in \Gr$ such that $z_1+w,\dots,z_k+w \in A$.
In particular, any set is $1$--universal and that is why we always have $\un (A) \ge 1$. 
It is easy to see, that for any $A\neq \Gr$ one has $\un (A) <\infty$ and, by convention, we put 
$\un (\Gr) = \infty$.
Also, if $B\subseteq A$, then $\un (B) \le \un (A)$. 
Finally, we have 
\[
    \un(A) = \un (A+x) = \un (\la \cdot A)
\]
for any $x\in \Gr$ and $\la\in \Z$ such that $(\la, |\Gr|) =1$. 
It is easy to observe (take a random intersection $A_X$) that a trivial upper bound for $\un (A)$, where $|A|=\d N$ is 
\begin{equation}\label{f:un_upper}
    \un (A) \le \frac{\log N}{\log (1/\d)} 
\end{equation}
and this bound is tight in general, see Section \ref{sec:covering}. 

\bp 

In paper \cite[inequality (29)]{SS_higher} it was  proved, in particular, that
for any $k$--universal set $U$ and an arbitrary set $S$ one has 
\begin{equation}\label{f:universal_addition}
    |U+S| \ge N \cdot \left( \frac{|S|}{N} \right)^{1/k} \,.
\end{equation}
The last bound is a consequence of the formula $U^{k-1} - \Delta_{k-1} (U) = \Gr^{k-1}$ and inequality \eqref{f:gen_triangle_S} of Lemma \ref{l:gen_triangle_S}, namely, 
\begin{equation}\label{f:gen_triangle}
    |S| N^{k-1} = |S| |U^{k-1} - \Delta_{k-1} (U)| \le |U^k - \Delta_k (S)| \le |U-S|^{k} \,.
\end{equation}
In particular, taking any $S$ with $|S|=1$, we obtain  
\begin{equation}\label{f:universal_size}
    |U| \ge N^{1-1/k} \,.
\end{equation}

\bp 

We want to study the connection between sumsets and universal sets. 
Probably, the first result in this direction 
was obtained by 
N.G. Moshchevitin (see paper \cite[Proposition 15]{SS_higher}, where the reader can find some additional results on this topic).

\begin{lemma}
    Let $k_1,k_2$ be positive integers, and $X_1,\dots,X_{k_1},Y$, $Z_1,\dots,Z_{k_2},W$ be finite subsets of an abelian group.
    Then
    $$
        |X_1 \m \dots \m X_{k_1} - \Delta_{k_1} (Y)| |Z_1 \m \dots \m Z_{k_2} - \Delta_{k_2} (W)|
            \le
    $$
    $$
            \le
                | (X_1 - W) \m \dots \m (X_{k_1}- W) \m (Y-Z_1) \m \dots \m (Y-Z_{k_2}) - \Delta_{k_1+k_2} (Y-W)| \,.
    $$
    In particular, for any non--empty sets $A,B \subseteq \Gr$ one has 
\begin{equation}\label{f:Moshchevitin}
    \un (A+B) \ge \un(A) + \un (B) - 1 \,.
\end{equation}
\label{l:Moshchevitin}
\end{lemma}


We improve estimate \eqref{f:Moshchevitin} significantly, see Corollary \ref{c:un_prod} below. First of all, let us obtain a simple lemma.

\begin{lemma}
    Let $A,B,X \subseteq \Gr$ be sets, and $X=\{x_1,\dots,x_t\}$.
    Put $D=A-B$ and suppose that $B_i$ are subsets of $B$, $i\in [t]$. 
    Then 
\begin{equation}\label{f:new_inclusion}
    A_X \subseteq D_{\bigcup_{i=1}^t (B_i+x_i)} \,.
\end{equation}
    In particular,
\begin{equation}\label{f:new_inclusion2}
    A_X \subseteq D_{B+X} \,.
\end{equation}
\label{l:new_inclusion}
\end{lemma}
\begin{proof}
    Let $z\in A_X \neq \emptyset$. 
    By definition $z-x_i \in A$ for all $i\in [t]$. 
    Also, let $b_i\in B_i$.
    Then $z\in D + b_i + x_i$ iff 
    $z-x_i-b_i \in A-b_i \subseteq D$, $i\in [t]$. 
    Thus, we have obtained \eqref{f:new_inclusion} and putting $B_i=B$, we get \eqref{f:new_inclusion2}. Another way to prove \eqref{f:new_inclusion2} is to use a particular case of the main inclusion \eqref{f:KK_inclusion}, namely, $A\subseteq D_B$. 
This completes the proof.
$\hfill\Box$
\end{proof}

\bp 

\begin{corollary}
    Let $A,B \subseteq \Gr$. 
    Then 
\begin{equation}\label{f:un_prod}
    \un (A+B) \ge \un (A) \un (B) \,.
\end{equation}
\label{c:un_prod}
\end{corollary}
\begin{proof}
    Let $D=A-B$, $m= \un (A)$, $n=\un (B)$.
    Take any $s_1,\dots, s_m \in \Gr^n$.
    By assumption the set $B$ is  $n$--universal set and hence any vector $s_i \in \Gr^n$, $i\in [m]$ can be written as 
    $s_i = b_i + \Delta_n (x_i)$, where $x_i \in \Gr$ and 
    $b_i = (b^{(1)}_i, \dots, b^{(n)}_i)$.
    Let $B_i = \{ b^{(1)}_i, \dots, b^{(n)}_i \}$, $i\in [m]$ and $X=\{x_1,\dots,x_m\}$.
    Using inclusion \eqref{f:new_inclusion} of Lemma \ref{l:new_inclusion}, we obtain 
\begin{equation}\label{tmp:19.03_1}
    A_X \subseteq D_{\bigcup_{i=1}^m (B_i+x_i)} = D_{s_1,\dots,s_m} \,.
\end{equation}
    Notice that the vector $(s_1,\dots,s_m)$ belongs to $\Gr^{mn}$.
    By assumption $A$ is $m$--universal set and hence $A_X \neq \emptyset$. 
    It means that $D$ is $mn$--universal set. 
This completes the proof.
$\hfill\Box$
\end{proof}

\bp 

Let us show that bound \eqref{f:un_prod} is tight. 

\begin{example}
\label{exm:universal_H}
    Let $\Gr = \F_2^n$ and $\bigsqcup_{i=1}^k S_i = [n]$, $|S_i| \sim n/k$ and let $$
    H_i = H_i (S_i) = \{ x=(x_1,\dots,x_n) \in \Gr ~:~ x_j = 0\,, \forall j\in S_i \} \,.
    $$
    Then each $H_i$ is a subspace of $\Gr$, $|H_i| \sim N^{1-1/k}$ and it is easy to check  (or see \cite[Theorem 2.1]{alon2009discrete}) that $U=\bigcup_{i=1}^k H_i$ is $k$--universal set, $|U| \sim k N^{1-1/k}$. 
    Let $U' = \bigcup_{i=1}^{k'} H'_i$, where subspaces $H'_i = H_i (S'_i)$ and the sets $S'_i$, $|S'_i| \sim n/k'$ form a random splitting of $[n]$. 
    Then it is easy to check that the sumset $U+U' = \bigcup_{i=1,j=1}^{k,k'} (H_i+H'_j)$ has size $O(kk'N^{1-1/kk'})$ and hence in view of \eqref{f:universal_size} one has $\un(U+U') \le \un(U) \un (U') = kk'$ for $(kk')^2 \log (kk') \ll \log N$, say. 
    A similar construction can be found in \cite{KLSS}.
\end{example}



\bp

Now we want to obtain a quantitative (or in other words, statistical) version of universality, which allows us to generalize   
Corollary \ref{c:un_prod}.
Given a set $A\subseteq \Gr$ and a positive integer $n$ define 
\[
    \U_n (A) = \U_n (A, \Gr) := \frac{|A^n-\D_n (\Gr)|}{N^n} \le 1 \,.
\]
In other words, $\U_n (A)$ is the proportion of $n$--tuples from $\Gr$ whose shifts belong to $A$.
In particular, 
$\U_n (A) = 1$ iff $A$ is $n$--universal. 
Sometimes we write $\ov{\U}_n (A) = \U^{1/n}_n (A)$.
In this terms, the inequality \eqref{f:gen_triangle} above can be rewritten as 
\begin{equation}\label{f:gen_triangle'}
    |A+S| \ge \left( \frac{|S|}{N} \cdot \U_n (A) \right)^{1/n} N = \left( \frac{|S|}{N} \right)^{1/n} \ov{\U}_n (A) N \,,
\end{equation}
    where $n$ be a positive integer and $A$ is an arbitrary  set. 
Further we know that (see \cite{SS_higher})
\begin{equation}\label{f:AAk_mult}
    |A^{nm} - \D_{nm} (A)| \le |A^n - \D_n (A)|^m \,.
\end{equation}
In terms of the quantity 
$\U_n$ 
the last estimate can be rewritten as 
\begin{equation}\label{f:AAk_mult_U}
    \U_{nm+1} (A) \le  \U^m_{n+1} (A) \,,
\end{equation}
and therefore 
$\{ \ov{\U}_n (A) \}_{n\ge 1}$ is decreasing by a subsequence. 
%
%
%
In this paper we are interested in {\it lower} bounds for the quantity 
$\U_n$.
To generalize Corollary \ref{c:un_prod} we need several new concepts. Given a vector $\vec{v} = (v_1,\dots, v_l)$, we write $|\vec{v}|=l$ for the number of components of $\vec{v}$. 

\begin{definition}
    Let $m,n$ be positive integers, $A \subseteq \Gr^{m}$, $B \subseteq \Gr^n$, and $m_1+\dots+m_n = m$,  where $0<m_j \in \Z$, $j\in [n]$.
    Define the following subset of $\Gr^{m+n}$
\[
    A \times_{m_1,\dots,m_n} B = \{ (a_1,b_1,\dots,a_n,b_n) ~:~ (a_1,\dots,a_n)\in A\,, (b_1,\dots,b_n) \in B,\, |a_i| = m_i,\, \forall i\in [n] \}\,.
\]
    Also, put
\[
    \Delta_{m_1,\dots,m_n;n} (B) = \{ (\D_{m_1} (b_1), \dots, \D_{m_n} (b_n)) ~:~ (b_1,\dots,b_n) \in B \} \subseteq \Gr^{m} \,.
\]
\end{definition}

\bp

    Let us make some useful remarks.
    First of all, we have $|\Delta_{m_1,\dots,m_n;n} (B)| = |B|$ and for any $x\in \Gr^n$ one has 
\begin{equation}\label{f:gen_Delta_shift}
    \Delta_{m_1,\dots,m_n;n} (B+x) = \Delta_{m_1,\dots,m_n;n} (B) + \Delta_{m_1,\dots,m_n;n} (x) \,.
\end{equation}
    Also, notice that  formula \eqref{f:gen_triangle_G} gives us for any $A\subseteq \Gr$
\begin{equation}\label{f:G_reduced_gen} 
    |A^m - \Delta_{m_1,\dots,m_n;n} (\Gr^n)| 
    = N^n \prod_{i=1}^n |A^{m_i-1} - \D_{m_i-1} (A)| \,.
\end{equation}
    Finally, the corresponding Cauchy--Schwarz inequality that connects our new difference set and a kind of energy is the following 
\[
    |Q|^{2m} |B|^2 \le |Q^m-\Delta_{m_1,\dots,m_n;n} (B)| 
\]
\begin{equation}\label{f:new_diff_energy} 
    \times 
    \sum_{b=(b_1,\dots,b_n),\, b'= (b'_1,\dots,b'_n) \in B}  (Q\circ Q)^{m_1} (b_1-b'_1) \dots (Q\circ Q)^{m_n} (b_n-b'_n) 
    \,,
\end{equation}
    where $Q\subseteq \Gr$, $B\subseteq \Gr^n$ are arbitrary sets.

    \bigskip

    Now we are ready to obtain the general triangle inequality which contains Lemma \ref{l:gen_triangle_S} for $n=1$. 

\begin{lemma}
    Let $m,n$ be positive integers, $A \subseteq \Gr^m$, $B,C \subseteq \Gr^n$,
    and $m_1+\dots+m_n = m$, where $0<m_j\in \Z$, $j\in [n]$.
    Then 
\begin{equation}\label{f:gen_triangle_new1}
    |C| |A-\Delta_{m_1,\dots,m_n;n} (B)| \le |A \times_{m_1,\dots,m_n} B - \Delta_{m_1+1,\dots,m_n+1;n} (C)| \,,
\end{equation}
    and 
\begin{equation}\label{f:gen_triangle_new1.5}
    |A \times_{m_1,\dots,m_n} B - \Delta_{m_1+1,\dots,m_n+1;n} (C)| 
    =
    |A \times_{m_1,\dots,m_n} C - \Delta_{m_1+1,\dots,m_n+1;n} (B)| 
    \,. 
\end{equation}
    Besides 
 \begin{equation}\label{f:gen_triangle_new2}
   |C| |A \pm \Delta_{m_1,\dots,m_n;n} (B)|
   \le 
   |A \pm \Delta_{m_1,\dots,m_n;n} (C)| |B\pm C| \,.
\end{equation}
\label{l:gen_triangle_new}
\end{lemma}
\begin{proof}
    Let us take any element 
    $x:=(a_1-\D_{m_1} (b_1), \dots, a_n - \D_{m_n} (b_n)) \in A-\Delta_{m_1,\dots,m_n;n} (B)$, where $a=(a_1,\dots,a_n) \in A$, $|a_i|=m_i$, $i\in [n]$, and $b=(b_1,\dots,b_n) \in B$. 
    As always we assume that there is a unique pair $(a,b)$ representing a fixed element $x$ of $A-\Delta_{m_1,\dots,m_n;n} (B)$. 
    Now consider the map
\[
    (x,c) \to 
    (a_1,b_1,\dots,a_n,b_n) - 
    (\D_{m_1+1} (c_1), \dots, \D_{m_n+1} (c_n))
    \in 
    A \times_{m_1,\dots,m_n} B - \Delta_{m_1+1,\dots,m_n+1;n} (C) \,,
\]
    where $c=(c_1,\dots,c_n) \in C$. 
    Subtracting $a_i - \D_{m_i} (c_i)$ and $b_i - c_i$ for any $i\in [n]$, we see that the defined map is an injection. In gives us \eqref{f:gen_triangle_new1} and 
    bound \eqref{f:gen_triangle_new2} with minuses follows from the obtained estimate \eqref{f:gen_triangle_new1}. 
    To 
    get 
    \eqref{f:gen_triangle_new2} with pluses consider the magnification ratio 
\[
    R_A[B] = R_A[B] (m_1,\dots,m_n) :=\min_{\emptyset \neq X \subseteq B} \frac{|A + \Delta_{m_1,\dots,m_n;n} (X)|}{|X|} \,.
\]
    Repeating the Petridis argument \cite{Petridis} (combining with  formula \eqref{f:gen_Delta_shift}, also see \cite[Section 8]{SS_higher}), we obtain for any $C\subseteq \Gr^m$ that 
\[
    |A+\Delta_{m_1,\dots,m_n;n} (C+X)| \le R_A[B] \cdot |C+X| \,.
\]
    Combining the last formula with 
    standard arguments (see, e.g.,  \cite[Corollary 37]{SS_higher}), 
    one obtains
\[
     |C| |A + \Delta_{m_1,\dots,m_n;n} (B)|
   \le 
   |A + \Delta_{m_1,\dots,m_n;n} (C)| |B + C| 
\]
    as required. 
%
%
    Finally, 
    the map 
\[
    (a_1,b_1,\dots,a_n,b_n) - 
    (\D_{m_1+1} (c_1), \dots, \D_{m_n+1} (c_n))
\]
\[
    \to
    (a_1,c_1,\dots,a_n,c_n) - 
    (\D_{m_1+1} (b_1), \dots, \D_{m_n+1} (b_n))
\]
    is injection and hence \eqref{f:gen_triangle_new1.5} follows. 
This completes the proof.
$\hfill\Box$
\end{proof}

\bp

Lemma \ref{l:gen_triangle_new} implies the following quantitative version of Corollary \ref{c:un_prod}.

\begin{corollary}
    Let $m,n$ be positive integers, and $A,B \subseteq \Gr$. 
    Then 
\begin{equation}\label{f:A-B_universal}
    \U_{nm} (A+B) \ge \U_m (A) \U^m_n (B) \,.
\end{equation}
    In particular, for any set $S\subseteq \Gr$ one has 
\begin{equation}\label{f:A-B_universal_expansion}
    |A+B+S| \ge  
    \left( \frac{|S|}{N} \right)^{1/mn} \ov{\U}^{1/n}_m (A) \ov{\U}_n (B) N 
    \,,
\end{equation}
    and for any positive integer $l$ the following holds 
\begin{equation}\label{f:A-B_universal_sumsets}
    \ov{\U}_{m^l} (lA) \ge \left(\ov{\U}_m (A) \right)^{\frac{m^l-1}{m^l-m^{l-1}}}  \,.
\end{equation}
\label{c:A-B_universal}
\end{corollary}
\begin{proof}
    Let $D=A-B$, $\mathcal{A} = A^{m} - \D_m (\Gr)$, and $\mathcal{B} = B^{nm}$. 
    From the proof of Corollary \ref{c:un_prod} (see formula \eqref{tmp:19.03_1}) we know that 
\[
    |D^{nm} - \D_{nm} (\Gr)|
    \ge 
    \left| \bigcup_{(x_1,\dots,x_m) \in A^{m} - \D_m (\Gr)} \left( (B^n - \D_n (x_1)) \times \dots \times  (B^n - \D_n (x_m))\right) \right|
\]
\[
=
|\mathcal{B} - \D_{n,\dots,n;m} (\mathcal{A})| \,.
\]
    Using Lemma \ref{l:gen_triangle_new}, 
    we obtain 
\[
    |\mathcal{A}| |B^n - \D_n (\Gr)|^m 
    =
    |\mathcal{A}| |\mathcal{B} - \D_{n,\dots,n;m} (\Gr^m)|
    \le 
    |\mathcal{B} \times_{n,\dots,n} \Gr^m - 
    \D_{n+1,\dots,n+1;m} (\mathcal{A})|
\]
\[
    \le 
    |\mathcal{B} - \D_{n,\dots,n;m} (\mathcal{A})| N^m 
    \,.
\]
    Combining the last two bounds, we derive
\[
    |(A-B)^{nm} - \D_{nm} (\Gr)| \ge N^{-m} |A^{m} - \D_m (\Gr)|^{} |B^n - \D_n (\Gr)|^m 
    \,, 
\]
    and \eqref{f:A-B_universal} follows.

    To get \eqref{f:A-B_universal_expansion} we use the obtained bound and repeat the calculations from \eqref{f:gen_triangle}. It gives us 
\[
    N |A-B+S|^{nm} \ge |S| |D^{nm} - \D_{nm} (\Gr)| \ge N^{-m} |S| |\mathcal{A}| |B^n - \D_n (\Gr)|^m 
\]
    Finally, estimate \eqref{f:A-B_universal_sumsets} follows from \eqref{f:A-B_universal} by induction. 
This completes the proof.
$\hfill\Box$
\end{proof}

\bp 

Considering the simplest case  $m=n=2$ of \eqref{f:A-B_universal_expansion}, we derive
\[
    |A+B+S|^4 \ge |S||A-A||B-B|^2 \,.
\]
Of course, the real power of \eqref{f:A-B_universal}, \eqref{f:A-B_universal_expansion} comes when the parameters $m$ and $n$ are large.

\bp

The same argument gives us a new expanding property of universal sets.

\begin{corollary}
    Let $n,k$ be positive integers, $U\subseteq \Gr$ be a $k$--universal set, and $S\subseteq \Gr^n$ be an arbitrary set, $|S|=\sigma N^n$.
    Then 
\begin{equation}\label{f:universal_addition_new}
|U^{nk}-\Delta_{k,\dots,k;n} (S)| \ge \sigma N^{nk} 
    \,.
\end{equation}
    In particular, for any sets $A,S\subseteq \Gr$ and an arbitrary $m\ge 1$ the following holds
\begin{equation}\label{f:universal_addition_new2}
    |U+A+S| \ge N \left( \frac{|S|}{N} \right)^{1/km} \ov{\U}^{1/k}_m (A)
    \,.
\end{equation}    
\label{c:universal_addition_new}
\end{corollary}
\begin{proof}
    Let $m=nk$.
    Applying  estimate  \eqref{f:gen_triangle_new1} of Lemma \ref{l:gen_triangle_new}, we get 
\[
    |S| N^m = 
    |S| |U^m-\Delta_{k,\dots,k;n} (\Gr^n)| \le |U^m \times_{k,\dots,k} \Gr^n - \Delta_{k+1,\dots,k+1;n} (S)| 
    \le 
    N^n |U^{m}-\Delta_{k,\dots,k;n} (S)| 
\]
as required. 
    Bound \eqref{f:universal_addition_new2} follows from \eqref{f:A-B_universal_expansion}. 
This completes the proof.
$\hfill\Box$
\end{proof}

\section{Universality and sets avoiding solutions to linear equations}
\label{sec:Ek}

The 
first 
result of this section shows that universal sets (or, more generally, sets with large $\U_n$) always contain solutions of all linear equations. It can be considered as a combinatorial reformulation of the first part of Kelley--Meka proof \cite{kelley2023strong}.

\begin{proposition}
    Let $n\ge 3$ be an integer, $A\subseteq \Gr$ be a set, $|A| = \d N$ and $A$ has no solutions to the equation 
\begin{equation}\label{f:linear_eq}
    \a_1 x_1 + \dots + \a_n x_n = \beta \,, \quad \quad x_j \in A,\, j\in [n] \,,
\end{equation}
    where $\a_1,\dots,\a_n$ are coprime to $|\Gr|$.
    Then for $n=3$ one has 
\begin{equation}\label{f:linear_eq2_n=3}
    \un (A) \le \d^{-1} \log (1/\d) \,,
\end{equation}
    and for $n>3$, we obtain 
\begin{equation}\label{f:linear_eq2}
    \un (A) \le (2\log (1/\d) )^{\frac{1}{[n/2]-1}} \,.
\end{equation}
    Finally, 
    one has 
    $\ov{\U}_{m} (A) < 7/8$ for all $m$ greater than 
\begin{equation}\label{f:linear_eq3}
    (3\log (1/\d) )^{\frac{1}{[n/2]-1}} \,.
\end{equation}    
\label{p:linear_eq}
\end{proposition}
\begin{proof}
    By Corollary  \eqref{c:un_prod} we know that for $U = \a_1 \cdot A + \dots + \a_{n-2} \cdot A$ 
    one has $\un (U) \ge \un^{n-2} (A)$. 
    Since the set $A$ has no solutions to \eqref{f:linear_eq}, 
    it follows that 
\[
    N-|A|\ge |U+\a_{n-1} \cdot A| > N(1-\un^{-1} (U) \log (1/\d)) 
\]
    and in the case $n=3$ the result follows. 
    For $n>3$ let us write $n=l+l'+2$, where $l' \ge l=[n/2]-1$ and define 
    $U_1 = \a_1 \cdot A + \dots + \a_{l} \cdot A$, $U_2 = \a_{l+3} \cdot A + \dots + \a_{n} \cdot A$.
    Now observe that the sets $\a_{l+1} \cdot A+U_1$, $-(\a_{l+2} \cdot A+U_2)$
    both have sizes greater than $N/2$ and thus they intersect.  
    Thus 
    \eqref{f:linear_eq2} follows.

    Finally, assume that $\ov{\U}_{m} (A) \ge 7/8$, where $m$ is given by \eqref{f:linear_eq3}. 
    Using  inequality \eqref{f:A-B_universal_sumsets}, combining with \eqref{f:gen_triangle'}, we obtain 
\[
    |\a_{l+1} \cdot A +U_1| \ge \d^{1/m^{l}} \ov{\U}_{m^l} (U_1) N 
        \ge  
        \d^{1/m^{l}} \ov{\U}_{m} (A)^{\frac{m^l-1}{m^l-m^{l-1}}} N
        \ge 
        \d^{1/m^{l}} \ov{\U}_{m} (A)^{\frac{m}{m-1}} N
        \,.
\]
    and similar for $|\a_{l+2} \cdot A+U_2|$.
    Therefore by our choice of $m$, we get 
\[
    |\a_{l+1} \cdot A+U_1| > 3N/4 \cdot (1- m^{-l} \log(1/\d))  \ge N/2 \,.
\]
This completes the proof.
$\hfill\Box$
\end{proof}

\bp 

In the next lemma we show that the uniformity in the Kelly--Meka sense implies universality. Of course, one can relax condition \eqref{cond:Ek_norm} using the H\"older inequality but we do not need this step.

\begin{lemma}
    Let $\eps \in (0,1)$ be a real number, $k$ be a positive integer, $A\subseteq \Gr$ be a set, $|A|=\d N$ and 
\begin{equation}\label{cond:Ek_norm}
    \|f_A \|^{2l}_{\E_l} = \sum_{x} (f_A \circ f_A)^{l} (x) \le \eps^{2l} \d^{2l} N^{l+1} \,, \quad \quad \forall l\in [k] \,.
\end{equation}
    Then 
\begin{equation}\label{f:Ek_norm_U}
    \ov{\U}_k (A) > \frac{1}{1+\eps^2} \,.
\end{equation}    
\label{l:uniformity_implies_universality}
\end{lemma}
\begin{proof}
    We have $(A\circ A)(x) = \d^2 N + (f_A\circ f_A)(x)$ and hence 
\[
    \E_k (A) = \sum_{x} (A\circ A)^k (x) = \d^{2k} N^{k+1} + \sum_{l=2}^k \binom{k}{l} (\d^2 N)^{k-l} \sum_{x} (f_A\circ f_A)^l (x) = \d^{2k} N^{k+1} + \mathcal{E} \,.
\]
    Using our assumption \eqref{cond:Ek_norm}, 
    one obtains 
\[
    \mathcal{E} \le \d^{2k} N^{k+1} \sum_{l=2}^k \binom{k}{l} \eps^{2l} < \d^{2k} N^{k+1} ((1+\eps^2)^k -1) \,.
\]
    Combining the last two bounds with the Cauchy--Schwarz inequality (see, e.g., \eqref{f:new_diff_energy}), we get 
\[
    |A|^{2k} \le \E_k (A) |A^{k-1} - \D_{k-1} (A)| < \d^{2k} N^{k} (1+\eps^2)^k |A^{k} - \D_{k} (\Gr)| \,.
\]
    It follows that $\ov{\U}_k (A) > (1+\eps^2)^{-1}$. 
This completes the proof.
$\hfill\Box$
\end{proof}

\bp

Now let us obtain the main result of this section that 
shows 
how to use universality 
to control  the number of the solutions to any linear equation \eqref{f:E_k_linear_new_cond}.

\begin{theorem}
    Let $n\ge 4$ be an integer, $A\subseteq \Gr$ be a set, $|A| = \d N$ and $A$ has no solutions to the equation 
\begin{equation}\label{f:E_k_linear_new_cond}
    \a_1 x_1 + \dots + \a_n x_n = \beta \,, \quad \quad x_j \in A,\, j\in [n] \,,
\end{equation}
    where $\a_1,\dots,\a_n$ are coprime to $|\Gr|$.
    Also, let $\eps \in (0,1/8]$ be any number.
    Then there is $l \le (3\log (1/\d) )^{\frac{1}{[n/2]-1}}+1$ such that 
\begin{equation}\label{f:E_k_linear_new}
    \|f_A \|^{2l}_{\E_l} \ge \eps^{2l} \d^{2l} N^{l+1} \,.
\end{equation}
\label{t:E_k_linear_new}
\end{theorem}
\begin{proof}
    Assume that \eqref{f:E_k_linear_new} does not hold, 
    otherwise there is nothing to prove. 
    By assumption $\eps \in (0,1/8]$ and hence using Lemma \ref{l:uniformity_implies_universality}, we see that, in particular,  $\ov{\U}_k (A) > \frac{1}{1+\eps^2} >7/8$, where $k=(3\log (1/\d) )^{\frac{1}{[n/2]-1}} +1$. 
    This contradicts bound \eqref{f:linear_eq3} of Proposition \ref{p:linear_eq}.  
$\hfill\Box$
\end{proof}

\bp 

Theorem \ref{t:E_k_linear_new} 
allows us to 
say something about sets in $\F_p^s$, 
avoiding solutions to affine equations.

\begin{corollary}
    Let $p>2$  be a prime number, $n\ge 4$ be an integer, $A\subseteq \F_p^s$ be a set, $|A| = \d N$ and $A$ has no solutions to the equation 
\begin{equation}\label{f:linear_equations_cond}
    \a_1 x_1 + \dots + \a_n x_n = 0 \,, \quad \quad x_j \in A,\, j\in [n] \,,
\end{equation}
    where $\a_1,\dots,\a_n \in \F^*_p$ and $\a_1 + \dots + \a_n = 0$. 
    Then 
\begin{equation}\label{f:linear_equations}
    |A| \ll q^{s-s^{\frac{1}{5}-\frac{2}{25[n/2]-15}}} \,.
\end{equation}
\label{c:linear_equations}
\end{corollary}
\begin{proof}
    Let $\mathcal{L} (\d)= \log(2/\d)$, and $p_* = (3\mathcal{L} (\d))^{\frac{1}{[n/2]-1}} +1$.
    Applying 
    Theorem \ref{t:E_k_linear_new} with $\eps:=1/8$, we see that \eqref{f:E_k_linear_new} takes place. 
    Now it remains to use the arguments of \cite{BS_improvement}. 
    Namely, 
    applying 
    \cite[Lemma 6]{BS_improvement} with $\a=\d$, $\a_1=\a_2 = \d^{p_*}$, we find a subspace $V\le \F_p^s$  and $x\in \F_p^s$ such that $|A\cap (V+x)| \ge \d (1+\eps/2)$ and 
\[
    \mathrm{codim} (V) \ll_\eps  \mathcal{L}^2 (\d) \mathcal {L}^2 (\d^{p_*}) \ll p^2_* \mathcal{L}^4 (\d) \,.
\]
    Applying the usual density increment argument, 
    one obtains a contradiction after at most $\mathcal{L}(\d)$ steps and after some calculations 
    we arrive at  \eqref{f:linear_equations}.
This completes the proof.
$\hfill\Box$
\end{proof}

\bp

The same method works for general abelian  groups.
Indeed, using  Lemma 8 from  \cite{BS_improvement} instead of \cite[Lemma 6]{BS_improvement} and repeating the argument will reduce roughly two logarithms from the main result of \cite{BS_improvement} in the case of large $n>3$. Nevertheless, we do not include the final bound due to the fact that the better estimate is contained in \cite{Schoen_convex_eq} (also, see previous paper \cite{Kosciuszko_eq}).

\section{On covering numbers}
\label{sec:covering}

Let $\Gr$ be an abelian group and $A\subseteq \Gr$ be a set.
Further put $\Omega = \Omega(A) = A^c$. 
We have 
\begin{equation}\label{f:cov_Omega}
    \cov (A) 
    = \cov^{+} (A)= \min_{X\subseteq \Gr} \{ |X| ~:~ A+X = \Gr \} = \min_{X\subseteq \Gr} \{ |X| ~:~ \Omega_X = \emptyset \} \,.
\end{equation}
In other words, for any $Y$, $|Y|\le \cov (A)-1 := k$ one has $\Omega_Y \neq \emptyset$.
The last statement is equivalent to $\Omega^{k-1} - \Delta_{k-1} (\Omega) = \Gr^{k-1}$ and thus $\Omega$ is a $k$--universal set. 
In particular, we have our basic formula \eqref{f:cov_un} and using \eqref{f:un_upper} as well as \eqref{f:cov_un}, we obtain the second upper bound of \eqref{f:basic_bounds}.
Also, 
formulae \eqref{f:cov_un}, \eqref{f:Moshchevitin} give us 
\begin{equation}\label{f:Omega_nA}
    (\cov (\Omega(A)) - 1)^n +1  \le \cov (\Omega(nA)) \,.
\end{equation}
More generally, for any set $E\subseteq \Gr$ we put 
\begin{equation}\label{f:cov_Omega_E}
    \cov (A; E) = \min_{X\subseteq \Gr} \{ |X| ~:~ E \subseteq A+X \} \,.
\end{equation}
In particular, $\cov (A; \Gr) = \cov(A)$.  
Also,  if $A' \subseteq A$ and  $E' \subseteq E$, then $\cov (A; E) \le \cov (A'; E)$ and $\cov (A; E') \le \cov (A; E)$. 
Notice that 
\begin{equation}\label{f:cov_E_mult}
    \cov (A) = \cov (A; \Gr) \le \cov(A;E) \cdot \cov(E; \Gr) \,.
\end{equation}
Since 
\begin{equation}\label{f:Omega_formulae}
    \Omega (A+X) = \Omega (A)_X
    \quad \quad 
        \mbox{ and }
    \quad \quad 
    \Omega (A_X) = \Omega(A) + X \,,
\end{equation}
it follows that $E\subseteq A+X$ is equivalent to $\Omega(A)_X \subseteq \Omega (E)$. 
In other words,  putting 
$k=\cov (A; E)-1$ we see that for any $z_1,\dots,z_k \in \Gr$ there is $w\in E$ such that $z_1+w,\dots,z_k+w \in \Omega (A)$. 
Another consequence of \eqref{f:Omega_formulae}
that follows from 
\eqref{f:cov_un} is 
\begin{equation}\label{f:Omega_formulae_un}
    \un (A+X) = \cov(\Omega(A)_X) - 1
    \quad \quad 
        \mbox{ and }
    \quad \quad 
    \un (A_X) = \cov(\Omega(A) + X) - 1 \,.
\end{equation}
In a similar way, using the second formula of \eqref{f:Omega_formulae}, the basic identity  \eqref{f:cov_un}, and 
Corollary \ref{c:un_prod}, 
we obtain for any $A \subseteq  \Gr$ that 
\begin{equation}\label{f:cov_As}
    \cov(A_X) \ge (\cov(A)-1) (\cov (X^c) - 1) + 1\,.
\end{equation}
The last formula gives us an interesting connection between $\cov(A)$ and $\cov (A^c)$.

\bp

\begin{lemma}
    Let $A,B \subseteq \Gr$ be any sets.
    Then 
\begin{equation}\label{f:cov_A_Ac}
    |B| \cov(A+B) \ge \cov(A) \ge (\cov(A+B)-1)(\cov (B^c) - 1) +1 \,.
\end{equation}
    In particular, 
\begin{equation}\label{f:cov_A_Ac2}
    \cov(A) \ge (\cov(A-A)-1) (\cov (A^c) - 1) +1\,.
\end{equation}
\label{l:cov_A_Ac}
\end{lemma}
\begin{proof}
    Let $A+B+Z = \Gr$, where $|Z|=\cov (A+B)$. Using formulae \eqref{f:Omega_formulae}, one obtains $\Omega(A+B+Z) = \Omega (A)_{B+Z} = \emptyset$ and hence
$
    \cov(A) \le \cov(A+B) |B| 
\,.
$
    Now apply \eqref{f:cov_As} with $A=A+B$ and $X=(-B)$.
    It remains to notice that $A \subseteq (A+B)_{-B}$ and hence $\cov((A+B)_{-B}) \le \cov (A)$. 
This completes the proof.
$\hfill\Box$
\end{proof}

\bigskip

Now we are ready to obtain 
the  main result of this section.

\begin{theorem}
    Let $A,B,C,E \subseteq \Gr$ be sets,  $|A| = \a N$, $|B| = \beta N$, and $|C| = \gamma N$.
    Then 
\begin{equation}\label{f:cov_AB_E}
    \cov (A+B; E) \le 
    \a^{-1}
    \log \frac{|B-E|}{|B|} + 1 
    \,,
\end{equation}
    In particular, 
\begin{equation}\label{f:cov_AB}
    \cov (A+B) \le \frac{1}{\a} \log \frac{1}{\beta} + 1 \,.
\end{equation}
    Now if $\a \ge N^{-2/3} \log^2 N$, then 
\begin{equation}\label{f:cov_ABC}
    \cov_{C,\sqrt{\a/8}} (A+B+C) \ll \frac{1}{\a} \cdot  \log^2 \frac{2}{\a \gamma} \cdot \left(  \log \frac{2}{\beta} + \log^2 \frac{2}{\a \gamma} \right) \,.
\end{equation}
\label{t:cov_ABC}
\end{theorem}
\begin{proof}
    Let us start with \eqref{f:cov_ABC} since to it is harder  to prove and instead of $\cov_{C,\sqrt{\a/8}} (A+B+C)$ we consider the simpler quantity $\cov_{} (A+B+C)$. 
    Put $\Omega = (A+B+C)^c$.
    Further, let $k=\cov (A+B+C)-1$ and 
    let $X \subseteq \Gr$ be a set of size $ck/\log(1/\beta)$ which we will define later. 
    Here $c>0$ is an absolute small constant. 
    Thanks to \eqref{f:KK_inclusion}, we have 
\begin{equation}\label{f:universal_ABCX}
    (A+C) \cap (\Omega_X - B - X)  = \emptyset \,.
\end{equation}
    We know that $\Omega$ is  a $k$--universal set. 
    Then it is easy to see that the set $\Omega_X$ is $k_*:=[k/|X|]$--universal set. 
    Using \eqref{f:universal_ABCX} and taking the absolute constant $c$ to be sufficiently small, we get 
\begin{equation}\label{tmp:12_02_2024}
    |\Omega_X - B| \ge N \beta^{1/k_*} \ge N/2 \,.
\end{equation}
    Let $T=\Omega_X - B$. 
    Applying 
    the Fourier transform, we rewrite property \eqref{f:universal_ABCX} as 
\begin{equation}\label{tmp:12_02_2024_1}
    \frac{|A||C||T||X|}{N} \le N^{-1} \sum_{\chi \neq 1} |\FF{A} (\chi)| |\FF{C} (\chi)| |\FF{T} (\chi)| |\FF{X} (\chi)|
\end{equation}
\[
    \le 
    N^{-1} \sum_{\chi \notin \Spec_\eps (C)} |\FF{A} (\chi)| |\FF{C} (\chi)| |\FF{T} (\chi)| |\FF{X} (\chi)|
    +
    N^{-1} \sum_{\chi \in \Spec'_\eps (C)} |\FF{A} (\chi)| |\FF{C} (\chi)| |\FF{T} (\chi)| |\FF{X} (\chi)|
\]
\[
    = \sigma_1+\sigma_2 \,,
\]
    where $\eps \in (0,1]$ is a certain parameter. 
    Using the Parseval identity \eqref{F_Par} and estimate \eqref{tmp:12_02_2024},  we get 
\[
    \sigma_1 < \eps |C| |X| \sqrt{|A||T|} \le \frac{|A||C||T||X|}{2N} \,,
\]
    where we have taken $\eps = \sqrt{\a/8}$ to satisfy the last inequality. 
    In a similar way, 
    applying the Parseval identity again, one has $t:= |\Spec_\eps (C)| \ll (\a\gamma)^{-1}$. 
    To estimate $\sigma_2$ we choose our set $X$  randomly.
    By the  standard Chernoff--type inequality  see, e.g., \cite[Lemma 3.2]{alon2007large}, one has for any $|X| \le N^{2/3}$ that 
\[
    |\FF{X} (\chi)| \ll \sqrt{|X|} \log t \,,
\]
    where $\chi$ runs over the set $\Spec'_\eps (C)$.
    We have the condition $|X| \le N^{2/3}$, since otherwise our estimate \eqref{f:cov_ABC}  immediately follows  from the simple bound 
    \eqref{f:covering_f_eps_trivial} (which is a direct consequence  of  \cite[Lemma 3.2]{alon2007large}) 
    and our assumption $\a \ge N^{-2/3} \log^2 N$. 
    Thus, as above 
\[
    \sigma_2 \ll \sqrt{|X|} \log t \cdot |C| \sqrt{|A||T|} \le \frac{|A||C||T||X|}{2N} \,, 
\]
    provided $|X| \gg \a^{-1} \log^2 t$. 
    Returning \eqref{tmp:12_02_2024_1}, we obtain a contradiction and it means that 
\[
    \frac{\cov(A+B+C)}{\log(1/\beta)}
    \ll 
    \frac{k}{\log(1/\beta)} \ll  |X| \ll \a^{-1} \log^2 t
    \ll 
    \a^{-1} \log^2 (1/\a\gamma) 
    \,.
\]
Thus, we have obtained an upper bound for $\cov(A+B+C)$ which is even better than \eqref{f:cov_ABC}. 
But it is easy to check  that our construction gives 
a similar 
estimate for $\cov_{C,\sqrt{\a/8}} (A+B+C)$. 
Indeed, put $k=\cov_{C,\sqrt{\a/8}} (A+B+C)-1$ as above and let us obtain a good upper bound for $k$.
First of all, we know that for $|X| \gg \a^{-1} \log^2 t$ one has $|\FF{X*C} (\chi)| \le \eps |X| |C|$, $\forall \chi\neq 1$ thanks to our construction. Secondly, our set $X$ is a random one, thus it is taken with probability close to one, and hence $|\Omega^{k-1} -\Delta_{k-1} (\Omega)| \ge 2^{-1} N^{k-1}$, say. 
Here $\Omega = (A+B+C)^c$ as above. 
Applying \eqref{f:gen_triangle}, we get  
\begin{equation}\label{tmp:12_02_2024_2}
    |\Omega-B| \gg N \beta^{1/k_*} \gg N 
\end{equation}
    for a sufficiently small constant $c>0$. 
    Unfortunately, to get an analogue of \eqref{tmp:12_02_2024} we need to replace the set $\Omega$ with the set $\Omega_X$ in \eqref{tmp:12_02_2024_2}, and the later is not $k_*$--universal set in general (due to the fact that the set of shifts $X$ is rather specific). 
    Nevertheless,  take another random set $X'$, $|X'| = k_*$, $k_* |X| \le k$ and then in view of  
    \begin{equation}\label{tmp:12_02_2024_3}
    |\FF{X+X'} (\chi)|  = |\FF{X*X'} (\chi)| \ll \sqrt{|X||X'|} \log^2 t \le \eps |X||X'|\,, \quad \quad \forall 
    \chi \in \Spec'_\eps (C) \,,
    \end{equation}
    we derive that with high probability  $\Omega_{X+X'} \neq \emptyset$ 
    thanks to 
    the definition of the set $\Omega$ and the quantity $\cov_{C,\sqrt{\a/8}} (A+B+C)$. 
    Here we have used the fact that $|X||X'| \gg k$ and if the second inequality in \eqref{tmp:12_02_2024_3} does not hold, 
    then we obtain \eqref{f:cov_ABC} immediately. 
    In particular, for 
    any random set 
    $X'=\{x'_1,\dots, x'_{k_*}\}$ there is $z\in \Gr$ such that $z+x'_1, \dots, z+x'_{k_*} \in \Omega_X$ and thus $|\Omega^{k_*-1}_X -\Delta_{k_*-1} (\Omega_X)| \gg N^{k_*-1}$.
    It follows that $|\Omega_X-B| \gg N$ and we can repeat the argument above.

    Now
    to obtain \eqref{f:cov_AB} we use the previous argument with 
    $X=C=\{0\}$. 
    In particular, $\Omega = (A+B)^c$. 
    Then by estimate \eqref{tmp:12_02_2024} with $k_*=k := \cov(A+B) -1$ one has 
\begin{equation}\label{tmp:15.02_1}
    |\Omega - B| \ge N \beta^{1/k} > N (1-k^{-1} \log (1/\beta)) \,,
\end{equation}
    and the last quantity is greater than $|N| - |A|$, provided
\[
    \cov(A+B) =  k + 1 > \a^{-1} \log (1/\beta) + 1 \,.
\]

    Finally, 
    to get \eqref{f:cov_AB_E} we use  an analogue of inequality  \eqref{tmp:15.02_1} as well as estimates \eqref{f:universal_addition}, \eqref{f:gen_triangle}.
    Namely, for $k := \cov (A+B; E) -1$, taking into account estimate \eqref{f:gen_triangle_S}, 
    as well as the argument after \eqref{f:Omega_formulae}, we have the following 
\[
    |B| N^{k} = |B| |\Omega^{k} - \Delta_{k-1} (E)| \le |\Omega^{k} \times E - \Delta_k (B)| \le |\Omega-B|^{k} |E-B|  \,,
\]
    and thus
\[
     |\Omega-B| \ge N \left( |B|/|B-E| \right)^{1/k} \,.
\]
    It follows that 
\[
    \cov (A+B; E) =  k + 1 \le \a^{-1} 
    \log (|B-E|/|B|) 
    + 1 \,.
\]
This completes the proof.
$\hfill\Box$
\end{proof}

\bp

Clearly, estimate \eqref{f:cov_AB} is  
tight 
up to some constants. 
Indeed,  let $A=B$ be an arithmetic progression, $|A|=\a N$, then $\cov(A)\ge (2\a)^{-1}$ and, consequently, the dependence on $\a$ is correct.
On the other hand, if $B$ is any set such that $|B|=1$ and $A$ is a random set, then in view of \eqref{f:basic_bounds} (also, consult Example \ref{exm:random}) we 
see 
that  \eqref{f:cov_AB} is optimal. 
A similar construction works for estimate \eqref{f:cov_ABC}.
%
%
%
Also, notice that the proof of inequalities \eqref{f:cov_AB_E}, \eqref{f:cov_AB}
says that, actually,  $A+B+X = \Gr$ for all typical $X$, $|X| \gg \a^{-1} \log (1/\beta)$.


\bp

The argument above gives us an upper bound for the covering numbers of $A$ in terms of $A_B$. 
Notice that the condition $A_B \neq \emptyset$ in \eqref{f:conseq_sums1} means that $B$ belongs to $A^{|B|}-\Delta_{|B|} (\Gr)$ and the assumption $A+B \neq \Gr$ in \eqref{f:conseq_sums2} is equivalent to say that $B$ is contained in $\Gr \setminus (x-A)$ for some $x\in \Gr$.

\begin{corollary}
    Let $A \subseteq \Gr$ be a  set.
    Then 
\begin{equation}\label{f:conseq_sums1}
    \cov (A) \le \min_{B \subseteq \Gr ~:~ A_B \neq \emptyset}\,  \frac{N}{|A_B|} \log \frac{N}{|B|} + 1 \,,
\end{equation}
    and 
\begin{equation}\label{f:conseq_sums2}
    \cov (\Omega(A)) \le \min_{B \subseteq \Gr ~:~ A+B \neq \Gr}\, \frac{N}{N-|A+B|} \log \frac{N}{|B|}  + 1 \,.
\end{equation}
    In particular, if for a certain $B \subseteq \Gr$, $|B| = \beta N$ one has $|A+B|\le (1-\eps)N$, then
\begin{equation}\label{f:conseq_sums3}
     \cov (\Omega(A)) \le \eps^{-1} \log (1/\beta) + 1 \,.
\end{equation}
\label{c:conseq_sums}
\end{corollary}
\begin{proof}
    The first inequality follows from  \eqref{f:cov_AB} and the inclusion \eqref{f:KK_inclusion}. 
    The second bound is a consequence of the first one, it remains to notice that 
    $|\Omega(A)_B| = |\Omega(A+B)| = N - |A+B|$ thanks to \eqref{f:Omega_formulae}. 
    Formula \eqref{f:conseq_sums3} follows from \eqref{f:conseq_sums2}. 
This completes the proof.
$\hfill\Box$
\end{proof}

\bigskip 

From the estimate \eqref{f:conseq_sums3} it follows, in particular, that for $|A| \gg N$ from the condition $\cov (\Omega(A))$ tends to infinity, we have $|A-A|, |A+A| = (1-o(1)) N$.


\bp

In the next corollary we show that inequality \eqref{f:Ruzsa_intr} of Theorem \ref{t:Ruzsa_intr} can be generalized for some  subsets of $A-A$. 


\begin{corollary}
     Let $k$ be a positive integer, $A \subseteq \Gr$ be a set, and $|A|=\a N$.
     Put $D=A-A$. 
     Then for any $X\in A^k - \Delta_k (A)$ one has 
\begin{equation}\label{f:KK_cov}
    \cov(D_X) \le \frac{1}{\a^{}} \log \frac{N}{|A_X|} + 1\,.
\end{equation}
\label{c:KK_cov}
\end{corollary}
\begin{proof}
    It is well--known \cite{Katz-Koester} and easy to see that $A_X-A \subseteq D_X$.
    After that apply estimate \eqref{f:cov_AB} of Theorem \ref{t:cov_ABC}.  
$\hfill\Box$
\end{proof}

\bp 

Using the same method, something can be said concerning higher sumsets (also, see Section \ref{sec:concluding}).

\begin{corollary}
    Let $A,B,C,D \subseteq \Gr$ be sets, $|A|=\a N$, $|B|=\beta N$, $C=\gamma N$, and $|D| = \d N$.
    Then 
\begin{equation}\label{f:Delta(D)}
    \cov (A\times B + \D_2 (C+D)) \le \frac{1}{\a \gamma} \log \frac{1}{\beta \d} +1 \,.
\end{equation}   
\label{c:Delta(D)}
\end{corollary}
\begin{proof}
    We have 
\[
    A\times B + \D_2 (C+D) = \{ (a+c,c) ~:~ a\in A,\, c\in C\} + 
    \{(d,b+d) ~:~ b\in B,\, d\in D\} = \mathcal{X} + \mathcal{Y} \,.
\]
    Also, $|\mathcal{X}| = |A||C|$ and $|\mathcal{Y}| = |B||D|$. 
    Using Theorem \ref{t:cov_ABC} with the sets $\mathcal{X}, \mathcal{Y}$ and the group $\Gr^2$, we obtain the result. 
$\hfill\Box$
\end{proof}

\begin{question}
    Let $A\subseteq \Gr$, $|A| \gg N$ and $k\ge 2$ be an integer. Is it true that $\cov(A^k -\D_k (A)) \ll 1$? 
\end{question}

    Finally, we show that it is always possible to decrease the covering number of a set by adding an arithmetic progression. 

\begin{lemma}
    Let $A,B \subseteq \Gr$ be sets, $|B| = \beta N$ and $|B+B| \le K|B|$.
    Then 
\begin{equation}\label{f:union_cov}
    \cov(A\cup B) \ge 2^{-1} \min\left\{ \frac{1}{\beta K^3}, \max\left\{ \frac{\beta \cov (A)}{\log (2K^4)}, \frac{\cov(A)}{2K^4 \log (1/\beta)} \right\} \right\} \,. 
\end{equation}
    In particular, for any $A$ and an arbitrary $k_* \in [1,\cov(A)/\log \cov (A)]$ there is $B$ such that 
\begin{equation}\label{f:union_cov2}
    c k_* \le \cov(A\cup B) \le  C k_* \,,
\end{equation}
    where $c,C>$ are some absolute constants. 
\label{l:union_cov}
\end{lemma}
\begin{proof}
    Let $k=\cov (A)$, $\Omega = A^c$ and $\Omega' =  (A\cup B)^c \subseteq \Omega$. 
    Also, let $l=\cov(A\cup B)$. 
    We have for a certain $X$, $|X|=l$ that $\Omega'_X = \emptyset$. 
    It is easy to see that the last formula implies that $\Omega_X \subseteq B+X$. 
    We can assume that $l\le (2\beta)^{-1}$, since otherwise there is nothing to prove. 
    Putting  $D=(B+X)^c$, we see that $|D| \ge N/2$. 
    Consider a maximal disjoint family of sets  $2B+z$, $z\in Z$, belonging to the set $D$. By maximality, we have $D\subseteq (2B-2B) + Z$ and hence 
\begin{equation}\label{tmp:01.04_3}
    |Z| \ge \frac{|D|}{|2B-2B|} \ge  \frac{N}{2K^4|B|} 
\end{equation}
    thanks to inequality \eqref{f:Pl-R}. 
    It remains to check that $Z\neq \emptyset$. 
    But if $Z=\emptyset$, then for any $z\in \Gr$ one has $(2B+z) \cap (B+X) \neq \emptyset$ and therefore $|X|\ge \beta^{-1} K^{-3}$ by the Pl\"unnecke inequality. 
    Thus, we can assume that \eqref{tmp:01.04_3} takes place and 
    denoting $Y=Z+B$, $|Y|=|Z||B| \ge N/(2K^4)$,  we see that $(Y+B) \cap \Omega_X = \emptyset$. 
    But $\Omega_X$ is $[(k-1)/(l-1)]$--universal set and the argument of the proof of Theorem \ref{t:cov_ABC} gives us 
\[
    \frac{k}{2l} \le \left[ \frac{k-1}{l-1} \right] \le \max\{ \beta^{-1} \log (2K^4), 2K^4 \log (1/\beta) \} 
\]
    as required.

    It remains to get \eqref{f:union_cov2}.
    Take $B$ to be any difference set with $|B+B| \ll |B|$ and $|B| \sim N/k_*$.
    Then $\cov(A\cup B)\le \cov (B) \ll k_*$ by Theorem \ref{t:Ruzsa_intr}, say. 
    Using our assumption we see that $k_* \log k_* \ll k$ and hence by \eqref{f:union_cov} we obtain  $\cov(A\cup B) \gg \beta^{-1} \gg k_*$.
    This completes the proof.
$\hfill\Box$
\end{proof}

\section{Some examples and the sum--product phenomenon}
\label{sec:sum-product}

Now 
we 
give 
several 
examples of sets from $\F_p$ (and general abelian groups) having large covering numbers $\cov^{+}, \cov^\times$. 
We study these quantities simultaneously, thus some statements in  this 
section 
can be considered as part of the sum--product phenomenon. 
The first results in this direction were obtained in \cite{s_Fish}, where we derived (see \cite[Proposition 5.2]{s_Fish}) the following simple sum--product--type result 
\[
    \cov^{\times} (A-A) \le \cov^{+} (A) \,.
\]

Our  first example is well--known and uses the random construction. 

\begin{example}
\label{exm:random}
    Let $A\subseteq \Gr$ or $A\subseteq \F_p$ be a random set, each element of $A$ is taken independently at random with a certain probability $\d$.
    Then \cite[Theorem 4.1]{bollobas2011covering} the following holds $\cov^{+} (A) \sim \d^{-1} \log N$ and $\cov^{\times} (A) \sim \d^{-1} \log p$. 
    Generally, one can take a random subset of an interval $I$ in $\Gr$ and obtain a similar bound. 
    Thus, sets with small sumset can have very large covering number. 
    Finally, using the first  inequality from 
    \eqref{f:cov_A_Ac}
    one can see that $\cov^{+} (A+X), \cov^{\times} (AX)\ge (\d|X|)^{-1} \log p$ for any set $X\subseteq \Gr$. 
\end{example}

Unlike the previous example, the following set is deterministic and has large covering number. 

\begin{example}
\label{exm:R}
    Let $\Gr = \F_p$, where $p$ be a prime number and $\mathcal{R}$ be the set of quadratic residue in $\F_p$. 
    Then $\cov^{+} (\mathcal{R})$ is at least $\left(\frac{1}{2}+o(1) \right)\log_2 p$ see, e.g., \cite[Proposition 14]{SS_higher}. 
    The same method gives $\cov^{\times} (\mathcal{R}+s) \ge \left(\frac{1}{2}+o(1) \right)\log_2 p$ for any non--zero $s\in \F_p$. 
    In a similar way, it is easy to check that $\ov{\U}_k (\mathcal{R}) = 1-o(1)$ for $k=(1-o(1))\log_2 p$ and in view of 
    \eqref{f:un_upper} this estimate is optimal. 
\end{example}

Let us generalize the previous example. 

\begin{proposition}
    Let $p$ be a prime number, $A,B \subseteq \F_p$ be sets, $|A|=\a p$, $|B|=\beta p$ such that $A/B \neq \F_p$ and $s\in \F^*_p$.
    Then both 
    $\cov^{+} (\F_p \setminus (A/B))$ and  $\cov^{\times} (\F_p \setminus (A/B) + s)$ are 
    at least $\frac{\log p}{\log (1/\a\beta)} (1-o(1))$.\\ 
    Similarly, if $A+B\neq \F_p$, 
    then 
     $\cov^{+} (\F_p \setminus (A+B)^{-1})$ and $\cov^{\times} (\F_p \setminus ((A+B)^{-1}+s))$ 
    are both at least $\frac{\log p}{\log (1/\a\beta)} (1-o (1))$.
\label{p:R_gen}
\end{proposition}
\begin{proof}
    We start with the first part and 
    let us prove $k:=\cov^{+} (\Omega (A/B)) \ge \frac{\log p}{\log (1/\a\beta)} (1-o(1))$, the second inequality follows similarly.
    Let $f(x)$ counts the number of 
    solutions to $x=a/b$, where $a\in A$, $b\in B$.
    Using the multiplicative Fourier transform, the Parseval identity and Weil's bound for the sums over multiplicative characters, we get for different $x_j$, $j\in [k]$ that 
\[
    \sum_{z} f(z+x_1) \dots f(z+x_k) = \frac{1}{(p-1)^k} \sum_{\chi_1, \dots, \chi_k} \FF{f} (\chi_1) \dots \FF{f} (\chi_k) \sum_{z} \chi_1 (z+x_1) \dots \chi_k (z+x_k) 
\]
\begin{equation}\label{tmp:21.02_1}
    > 
    \frac{(|A||B|)^k}{(p-1)^{k-1}} - (k+1)\sqrt{p} (|A||B|)^{k/2} > 0 \,,
\end{equation}
    provided $k = \frac{\log p}{\log (1/\a\beta)} (1-o(1))$. 
    Inequality \eqref{tmp:21.02_1} implies that $(A/B)_{x_1,\dots,x_k} \neq \emptyset$.

    Now let us obtain the second part of our result and again it is suffices to estimate the quantity $k:=\cov^{+} (\Omega(A+B)^{-1})$, since  the second bound follows in a similar way.
    Putting $f=A*B$ and applying the additive Fourier transform, we obtain
\[
    \sum_{z} f( (z+x_1)^{-1}) \dots f( (z+x_k)^{-1}) = \frac{1}{p^k} \sum_{r_1, \dots, r_k} \FF{f} (r_1) \dots \FF{f} (r_k) \sum_{z} e(r_1 (z+x_1)^{-1} +  \dots + r_k (z+x_k)^{-1}) \,.
\]
    One can check (calculating the Wronskian, for example) that the functions $(z+x_1)^{-1}, \dots, (z+x_k)^{-1}$ are linearly independent. 
    After that we use the  Weil bound again 
    as well as 
    the previous argument. 
    In the case of $\cov^{\times} (\F_p \setminus ((A+B)^{-1}+s))$ our functions are $f((x_j z-s)^{-1})$ and, thus  we can actually repeat our approach. 
This completes the proof.
$\hfill\Box$
\end{proof}

\begin{question}
    Let $\Gamma \le \F^*_p$ be a multiplicative subgroup, $|\G| < p^{1-\eps}$, where $\eps>0$ is any constant. 
    Is it true that $\cov^{+} (\Gamma) \gg_\eps \frac{p\log |\G|}{|\G|}$? 
\end{question}

Now let us consider the complements to the sets from Proposition \ref{p:R_gen}.

\begin{question}
    Let $p$ be a prime number and $A,B \subseteq \F_p$ be sets such that $|A|,|B| \gg p$ and $|(A/B)^c| \gg p$, $|(A-B)^c| \gg p$. Is it true that $\cov^{+} (A/B)$ and  $\cov^{+} ((A-B)^{-1})$ are both $\Omega(\log p)$? For $A=B$? 
\label{q:cov+(A/B)}
\end{question}


We give a partial answer to the question above.

\begin{proposition}
    Let $p$ be a prime number and $A,B \subseteq \F_p$ be sets such that $|A|,|B| \gg p$ and 
    $|(AB)^c| \gg p$, $|(A+B)^c| \gg p$. 
    Then 
    $\cov^{+} (A \cdot_{\Omega(1)} B)$ and 
    $\cov^{+} ((A+_{\Omega(1)} B)^{-1})$ are 
    both 
    $\Omega(\log p)$. 
\label{p:R_gen_opp}
\end{proposition}
\begin{proof}
    Consider the case of $\cov^{+} ((A+_{\Omega(1)} B)^{-1})$, say, another bound follows similarly. 
    By assumption $|A|,|B| \gg p$ and $|(A+B)^c| \gg p$. 
    Hence Theorem \ref{t:croot2012some} implies that the complement to $A+_{\Omega(1)} B$ contains a translation $\mathcal{B}(d,\eps)+s$ of a Bohr set $\mathcal{B}(d,\eps)$, where $d=O(1)$ and $\eps\gg 1$. 
    In view of estimate \eqref{f:Bohr_size}, we see 
    that $|\mathcal{B}(d,\eps)| \gg p$.
    Clearly, $\mathcal{B}(d,\eps)+s$ contains $(\mathcal{B}(d,\eps/2)+s) + \mathcal{B}(d,\eps/2)$. Using Proposition \ref{p:R_gen}, we get 
\begin{equation}\label{tmp:01.04_2}
    \cov^{+} ((A+_{\Omega(1)} B)^{-1}) \ge \cov^{+} ( (((\mathcal{B}(d,\eps/2)+s) + \mathcal{B}(d,\eps/2))^{-1})^c ) \gg \log p
\end{equation}
as required. 
$\hfill\Box$
\end{proof}

\bp

A 
natural approach to 
Question \ref{q:cov+(A/B)} 
is to try to choose a random subset $S$ of $A+B$
(we consider the question about $\cov^{+} ((A+B)^{-1})$, say) with small Wiener norm  \eqref{def:Wiener} and repeat the proof of Proposition \ref{p:R_gen}. 
However, generally speaking, this is impossible, since the Wiener norm of an arbitrary  set tends to infinity \cite{Konyagin_Littlewood}, \cite{MPS_Littlewood}.
Anyway, 
the proof of  Proposition \ref{p:R_gen} gives us 
the following connection between the Wiener norm  of a set and its covering number.

\begin{corollary}
    Let be a prime number, $A \subseteq \F_p$ be a set such that $|A| = \a p$, $\|A\|_W = K$ and $s\in \F_p$ be an arbitrary non--zero number. Then $\cov^{+} (\Omega(A^{-1}))$ and  $\cov^\times (\Omega (A^{-1}+s))$ are both at least $\frac{\log p}{2\log (K/\a)} (1-o (1))$.
\end{corollary}

The next example is of a different type and provides a wide  family of sets with large $\cov^\times$.

\begin{example}
\label{exm:A-A}
    Let $\Gr = \F_p$, $A\subseteq \F_p$, $|A|=\a p$ be an arbitrary set such that $A-A \neq \F_p$ and consider $(A-A)^c$. We know that $A-A$ contains a multiplicative shift of any collection of points $x_1,\dots,x_k \in \F_p$, where $k\le \frac{\log (p-1)}{\log (1/\a)}$ 
    see, e.g., \cite[Corollary 6]{SS_higher}. 
    Thus $\cov^{\times} ((A-A)^c) \ge \frac{\log (p-1)}{\log (1/\a)}$ (for sets $A$ with small sumsets  consult \cite{CRS}). 
    For example, $\cov^{\times} ([p/3, 2p/3]) \gg \log p$ and 
    the last fact 
    was known before see, e.g.,  \cite[Section 5]{s_Fish}.
    In view of Lemma \ref{l:Fish_covering} one can see that in general $\cov^{\times} ((A+B)^c)$ can be small. 
\end{example}

The later example allows us to say something about the covering numbers of sets avoiding solutions to some linear equations.

\begin{corollary}
    Let $c\in \F_p$ be a nonzero number, $A \subseteq \F_p$, $|A|=\a p$ be a set having no solutions to the equation $x-y=cz$, where $x,y,z\in A$. 
    Then $\cov^\times (A) = \Omega (\log p /\log (1/\a))$. 
\end{corollary}


    Finally, let us show that there are universal sets of special type. We will need this example in the proof of Theorem \ref{t:table_intr}  from the introduction.

\begin{example}
\label{exm:universal_sumset}
    Given $1\le k \ll \log N$ let us construct a $k$--universal set $U \subset \Gr:= \Z/N\Z$ such that 
\[
    U=A+B\,, \quad \quad |U^c| \gg N \quad \quad 
    \mbox{and} \quad \quad |A|,\, |B| \gg N \,. 
\]
    Let $d_0 = \lceil \sqrt{N} \rceil$ and let $d$ be a prime $d_0 \le d \le 2d_0$, $d$ does not divide $N$. 
    It is easy to see that such $d$ exists.
    Now choose a set $S_0$, $|S_0|=cd$, where $c>0$ is a sufficiently small constant  such that $S_0$ 
    is an arbitrary 
    $k' = \Omega(k)$--universal subset of $\Z/d\Z$. 
    Such sets exist thanks to Example \ref{exm:random} or see paper \cite{alon2009discrete}. 
    Put $S = S_0 \cup (S_0 +d)$. 
    Further, any $x\in \{0,1, \dots,N-1\}$ can be written as $x=y+dz$, where $y,z\in \{0,1,\dots,d-1\}$.
    Given any $x_1,\dots,x_{k'} \in \Gr$ consider them as elements of $\{0,1, \dots,N-1\}$ and therefore we can write $x_j = y_j+dz_j$, where $y_j,z_j\in \{0,1,\dots,d-1\}$, $j\in [k']$. 
    Since $S_0$ is $k'$--universal set of $\Z/d\Z$, it follows that there are $w_1, w_2 \in  \{0,1, \dots,p-1\}$ such that $y_j+w_1 \in S$ and $z_j+w_2 \in S$ for all $j\in [k']$. 
    Put $w=w_1+dw_2$. 
    Thus $x_j + w \in S + d \cdot S$, $j\in [k']$ and hence $S + d \cdot S$ is $k'$--universal set in $\Gr$.
    Now let $P=[c_* N]$, where $c_*>0$ is another sufficiently small constant and choose $s\in \Gr$ in a random way.
    Then with high probability the set $Q = d^{-1}(P+s) \cap P$ has size $\Omega (|P|^2/N) = \Omega (N)$. 
    Put $A=S \cup d \cdot Q$, $B = d \cdot S \cup Q$. 
    Then $|A|,|B| \gg N$. 
    Finally, the set $U:=A+B$ is $\Omega(k)$--uniform and 
\begin{equation}\label{tmp:01.04_1}
    |U| \le |S+d\cdot S| + 2|S+P| + |P+P+s| \le |S|^2 + 4|P| + 2d < 3N/4 
\end{equation}
    as required. Here we have chosen our constants $c,c_*$ to be sufficiently small.

    Notice that it is possible to 
    take 
    $B=A$ or  $B=-A$.
    Indeed, in the first case  just choose $A=B= A_0 (S) := S \cup d \cdot Q_* \cup d \cdot S \cup Q_*$, where $Q_* = P \cap d^{-1} (P+s) \cap d (P+t)$ and  the shifts $s,t$ are taken in a random way. It gives us $|Q_*| \gg N$ and an analogue of \eqref{tmp:01.04_1} holds. 
    In the second case let $A=A_0 (S')$, $B=-A_0 (S')$, where $S'= \pm S_0 \cup \pm (S_0 +d)$, $S'=-S'$. 
\end{example}

Now we are ready to proof Theorem \ref{t:table_intr}.

\bp 

\begin{proof}
We need to say something about any element of our table, which corresponds to the matrix $2\times 8$. Denote this matrix by $M$. 
Theorem \ref{t:cov_ABC} allows us to fill cells 
$M_{11}, M_{13}, M_{52}$ by $O(1)$. 
Further, 
$M_{21}$ follows from  Lemma \ref{l:Fish_covering},  $M_{16}$, $M_{18}$, and $M_{15}$, $M_{17}$ follow from Propositions \ref{p:R_gen}, \ref{p:R_gen_opp}. 
Example \ref{exm:A-A} corresponds to cells $M_{22}$ and $M_{24}$. Similarly, one can have everything 
(i.e. $O(1)$ or $\Omega(L)$) 
in $M_{23}$ and $M_{27}$ thanks to Example \ref{exm:A-A} due to $[p/3,2p/3]$ is a sumset. Finally, the remaining cells $M_{12}, M_{14}, M_{26}$ are all $\forall_{O(1)}$. Indeed, by Theorem \ref{t:croot2012some} the popular sums of these sets contain a shift of a Bohr set $\mathcal{B}(d,\eps)$, $|\mathcal{B}(d,\eps)| \gg p$ (additive or multiplicative).
Thus, using Theorem \ref{t:cov_ABC} and calculation in \eqref{tmp:01.04_2}, we obtain that the corresponding covering numbers are $O(1)$. 
On the other hand, the usual sumsets in $M_{12}, M_{14}, M_{26}$ and in $M_{28}$ 
can have $O(1)$ or $\Omega(L)$ as the covering numbers. 
Indeed, consider, for example, $M_{12}$. Then $\cov^{+}((A-A)^c)$ is, basically, $\un^{+} (A-A)$ and in the light of Example \ref{exm:universal_sumset} it can be $\Omega(L)$ and if $A$ is a sufficiently small  arithmetic progression, then $\cov^{+} ((A-A)^c) \ll 1$. 
The same is true for $M_{14}, M_{26}$. 
As for $M_{28}$ 
one has 
\[
\cov^\times (((A+B)^{-1})^c) = \un^\times ((A+B)^{-1}) +1 = \un^\times (A+B) + 1
\]
and thus we arrive at the cell $M_{24}$. 
This completes the proof of Theorem \ref{t:table_intr}.
$\hfill\Box$
\end{proof}

\section{Concluding remarks}
\label{sec:concluding}

   Universal sets also have other very strong properties. For example, let $U \subseteq \Gr$ be a (one--dimensional) $k$--universal set. 
    Then by estimate \eqref{f:universal_addition} we know that $U$ has the repulsion property, namely, for 
    an arbitrary set $S$ one has 
\begin{equation}\label{f:universal_addition'}
    |U+S| \ge N \cdot \left( \frac{|S|}{N} \right)^{1/k} \,.
\end{equation}
    Now let $m$ be a positive integer, $U_1,\dots, U_m$ be $k$--universal sets and let us consider the following multi--dimensional set  $\mathcal{U}:= U_1 \times \dots \times U_m  
    \subseteq \Gr^m$. 
    One can check 
    that 
    \[
        \mathcal{U}^{k} - \D_k (\Gr^m) = (U^k_1 - \Delta_k (\Gr)) \times \dots \times (U^k_m - \Delta_k (\Gr)) = \Gr^{mk}
    \]
    and thus $\mathcal{U}$ is a $k$--universal set as well. 
    One can rewrite the last formula as 
\begin{equation}\label{f:un_Cartesian_product}
    \un (U_1 \times \dots \times U_m, \Gr^m) = \min_{j\in [m]} \un (U_j, \Gr) \,.
\end{equation}
    In particular, we have an analogue of inequality \eqref{f:universal_addition'} for the set $\mathcal{U}$. 
    It is easy to see that bound \eqref{f:universal_addition'}
    can be generalized and refined for the sets of the form $\Delta_m (S)$. 
    Indeed,  
    let $k-m<k_*\le k$ be the maximal  number not exceeding $k$ such that $m$ divides $k_*$. 
    Write $d=k_*/m \in \Z$.
    Then as in \eqref{f:gen_triangle}, we get 
\[
    |S| N^{k_*-1} = |S| |U^{k_*-1} - \Delta_{k_*-1} (U)| \le |U^{k_*} - \Delta_{k_*} (S)| \le |U^m-\Delta_m (S)|^{d}   \,.
\]
    Thus if $|S| = \sigma N$, then 
\[
    |U^m-\Delta_m (S)|\ge N^m \sigma^{1/d} 
    > N^m (1-d^{-1} \log (1/\sigma) ) \,.
\]

This argument gives us the following result about universal sets and hence about the complements of sets with the large covering number, see examples of Section \ref{sec:sum-product}.
In particular, it is possible to obtain an analogue of Corollary \ref{c:conseq_sums} for higher sumsets.

\begin{corollary}
    Let $U\subseteq \Gr$ be a $k$--universal set, and $S\subseteq \Gr$ be an arbitrary set, $|S|=\sigma N$.
    Then 
\begin{equation}\label{f:universal_addition_Um}
|U^m-\Delta_m (S)| 
    > N^m \left( 1- \frac{m}{k-m+1} \log (1/\sigma) \right) \,.
\end{equation}
    In particular, for any $\mathcal{A} \subseteq \Gr^m$, $|\mathcal{A}| = \a N^m$ one has 
\[
    (\mathcal{A} + \Delta_m (S)) \cap U^m \neq \emptyset \,,
\]
    provided $\a \ge \frac{m}{k-m+1} \log (1/\sigma)$. 
\label{c:universal_addition_Um}
\end{corollary}

    Thanks to formulae \eqref{f:characteristic2} of Lemma \ref{l:characteristic} we have, in particular, that 
\[
    \sum_{x\in U-S} |U-(S\cap (U-x))| = |U^2 - \D_2 (S)| > N^2 \left( 1- \frac{2 \log (1/\sigma)}{k-1} \right) \,,
\]
and 
\[
    \sum_{x\in U-S} |U-(S\cap (U-x))|^2 \ge \sum_{x\in U-S} |U^2-\D_2(S\cap (U-x))| = |U^3 - \D_3 (S)| >
    N^3 \left( 1- \frac{3 \log (1/\sigma)}{k-2} \right) \,.
\]
Thus, for any sufficiently dense set $S\subseteq \Gr$ there are many shifts of the universal set $U$ such that $U-(S\cap (U-x))$ is close to the 
entire
group $\Gr$.

\bp 

It would be interesting to find further properties of universal sets.




\bibliographystyle{abbrv}

\bibliography{bibliography}{}

\end{document}